\documentclass[a4paper,11pt]{amsart}

\usepackage{amsfonts,amssymb}
\usepackage{textcomp}

\theoremstyle{plain}

\newtheorem{thm}{Theorem}[section]

\newtheorem{pro}[thm]{Proposition}
\newtheorem{lem}[thm]{Lemma}
\newtheorem{cor}[thm]{Corollary}
\newtheorem{exa}[thm]{Example}
\newtheorem{rmk}[thm]{Remark}

\newcommand{\Z}{\mathbb{Z}}

\newcommand{\N}{\mathbb{N}}

\DeclareMathOperator{\PAut}{PAut}
\DeclareMathOperator{\Inn}{Inn}
\DeclareMathOperator{\Cr}{Cr}
\DeclareMathOperator{\Aut}{Aut}

\DeclareMathOperator{\Dih}{Dih}

\def\l{\langle}
\def\r{\rangle}

\begin{document}

\title[A finiteness condition on centralizers]%
{A finiteness condition on centralizers\\ in locally finite groups}

\author[G.A. Fern\'andez-Alcober]{Gustavo A. Fern\'andez-Alcober}
\address{Matematika Saila\\ Euskal Herriko Unibertsitatea UPV/EHU\\
48080 Bilbao, Spain. {\it E-mail address}: {\tt gustavo.fernandez@ehu.eus}}

\author[L. Legarreta]{Leire Legarreta}
\address{Matematika Saila\\ Euskal Herriko Unibertsitatea UPV/EHU\\
48080 Bilbao, Spain. {\it E-mail address}: {\tt leire.legarreta@ehu.eus}}
%\email{leire.legarreta@ehu.es}

\author[A. Tortora]{Antonio Tortora}
\address{Dipartimento di Matematica\\ Universit\`a di Salerno\\
Via Giovanni Paolo II, 132\\ 84084 Fisciano (SA)\\ Italy. {\it E-mail address}: {\tt antortora@unisa.it}}
%\email{antortora@unisa.it}

\author[M. Tota]{Maria Tota}
\address{Dipartimento di Matematica\\ Universit\`a di Salerno\\
Via Giovanni Paolo II, 132\\ 84084 Fisciano (SA)\\ Italy. {\it E-mail address}: {\tt mtota@unisa.it}}
%\email{mtota@unisa.it}

\thanks{The first two authors are supported by the Spanish Government, grants
MTM2011-28229-C02-02 and MTM2014-53810-C2-2-P, and by the Basque Government, grant IT753-13. The last two authors would like to thank the Department of Mathematics at the University of the Basque Country for its excellent hospitality while part of this paper was being written; they also wish to thank G.N.S.A.G.A. (INdAM) for financial support.}

\keywords{Centralizers, locally finite groups\vspace{3pt}}

\begin{abstract}
We consider a finiteness condition on centralizers in a group $G$, namely that $|C_G(x):\langle x \rangle|<\infty$ for every $\langle x \rangle \ntriangleleft G$. For periodic groups, this is the same as $|C_G(x)|<\infty$ for every $\langle x \rangle \ntriangleleft G$. We give a full description of locally finite groups satisfying this condition. As it turns out, they are a special type of cyclic extensions of Dedekind groups. We also study a variation of our condition, where the requirement of finiteness is replaced with a bound: $|C_G(x):\langle x \rangle|\le n$ for every $\langle x \rangle \ntriangleleft G$, for some fixed $n$. In this case, we are able to extend our analysis to the class of periodic locally graded groups.
\end{abstract}

\maketitle

\section{Introduction}

In the present paper we focus on a natural finiteness condition on centralizers of elements in a locally finite group. Our motivation stems from our previous works \cite{FLTT} and \cite{FLTT2}, where we dealt with two restrictions on centralizers in finite $p$-groups, namely that $|C_G(x):\langle x \rangle|\le n$ either for all $x\in G\smallsetminus Z(G)$, or for all
$x\in G$ such that $\langle x \rangle \ntriangleleft G$. Actually, in \cite{FLTT}, we studied the first of these conditions in the case of general finite groups.

Given a group $G$, we say that $G$ is an FCI-group (FCI for `finite centralizer index') provided that
\begin{equation}
\label{FCI}
|C_G(x):\langle x \rangle| < \infty
\quad
\text{for every $\langle x \rangle \ntriangleleft G$.}
\end{equation}
For periodic groups this condition is the same as the simpler
\begin{center}
\label{finite centralizer}
$|C_G(x)| < \infty
\quad
\text{for every $\langle x \rangle \ntriangleleft G$.}$
\end{center}
We can then ask for the existence of a uniform bound for the indices in (\ref{FCI}), thus getting the condition that, for some positive integer $n$,
\begin{equation}
\label{BCI}
|C_G(x):\langle x \rangle| \le n
\quad
\text{for every $\langle x \rangle \ntriangleleft G$,}
\end{equation}
in which case we speak of BCI-group (BCI for `bounded centralizer index').

It is also natural to look at the same condition as in (\ref{FCI}) with
non-central elements instead of elements generating a non-normal subgroup, which reads
\begin{equation}
\label{strong FCI}
|C_G(x):\langle x \rangle| < \infty
\quad
\text{for every $x\in G\smallsetminus Z(G)$.}
\end{equation}
However, this latter condition has such a strong implication in our context that it does not make sense to consider it. In fact,
it is easy to see that every periodic group containing an infinite abelian subgroup, and in particular every infinite locally finite group, is abelian whenever (\ref{strong FCI}) is satisfied.

\vspace{8pt}

There are well-known results that give information about the structure of a group simply from the knowledge that the centralizer of one element is finite.
For example, Shunkov \cite[page 263]{Sh} proved that a periodic group having an involution with finite centralizer is necessarily locally finite.
This property does not extend to non-involutions, but by combining results of Hartley and Meixner
\cite{HM}, Fong \cite{Fo} and Khukhro \cite{Kh} with the classification of finite simple groups, it follows that a locally finite group having an element of prime order with finite centralizer has a nilpotent subgroup of finite index \cite[Corollary 5.4.1]{Kh}. See also \cite{Shu} for a more recent account.

In (\ref{FCI}) and (\ref{BCI}), we are imposing conditions on a broad set of centralizers and, as a consequence, the results that we obtain are much more precise. Indeed, in Theorem \ref{determination locally finite FCI-groups}, we will completely determine the structure of infinite locally finite FCI-groups. More precisely, we will show that such a group is either a Dedekind group, or can be constructed as an extension of a certain infinite periodic Dedekind group $D$ by an appropriate power automorphism of $D$ of finite order. The fact that Dedekind groups arise in this context is not surprising, since they vacuously satisfy the FCI-condition. As a consequence of Theorem \ref{determination locally finite FCI-groups} we will obtain that, in the realm of locally finite groups, FCI- and BCI-conditions are equivalent (Corollary \ref{locally finite FCI->BCI}).

We will also deal with locally graded periodic BCI-groups, by showing that such groups are locally finite (Theorem \ref{graded}).
Recall that a group is \emph{locally graded\/} if every non-trivial finitely generated subgroup has a non-trivial finite image. The class of locally graded groups is rather wide, since it includes locally (soluble-by-finite) groups, as well as residually finite groups, and so in particular free groups. The restriction to locally graded groups is motivated to avoid Tarski monster groups, that is, infinite (simple) $p$-groups, for $p$ a prime, all of whose proper non-trivial subgroups are of order $p$. Obviously, Tarski monster groups are BCI-groups.

It remains unclear whether Theorem \ref{graded} holds for all FCI-groups. This is related to the following question:
{\em Given a periodic residually finite group $G$ in which the centralizer of each non-trivial element is finite, is $G$ locally finite?}
Observe that the well-known examples of finitely generated infinite periodic groups constructed by Golod \cite{Go}, Grigorchuk \cite{Gr}, and Gupta-Sidki \cite{GS} are residually finite but not FCI-groups (see \cite[Theorem 1.1]{KS}, \cite[Theorem 1]{Roz} and \cite[Theorem 2]{Sid}, respectively).

\vspace{8pt}

\noindent
\textit{Notation\/}. We use mostly standard notation in group theory. In particular, if $G$ is a finite $p$-group, we denote by $\Omega_1(G)$ the subgroup generated by all elements of $G$ of order $p$.
If $G$ is a periodic group, we write $\pi(G)$ for the set of prime divisors of the orders of the elements of $G$.
Furthermore, if $G$ is periodic and
nilpotent, $G_p$ denotes the unique Sylow $p$-subgroup of $G$, and if $\varphi$ is an automorphism of $G$, then $\varphi_p$ stands for the restriction of $\varphi$ to $G_p$.
On the other hand, we use $\prod$ to denote the direct product (i.e.\ restricted cartesian product) of a family of groups, and $\Cr$ for the unrestricted cartesian product.  Finally, we write $\Z_p$ for the ring of $p$-adic integers.
Also, if $R$ is a ring, we let $R^{\times}$ denote the group of units of $R$.

\section{Examples and general properties}

In this section we present some examples and general properties of FCI- and BCI-groups.

A prototypical source of FCI-groups comes from the family of the so-called generalized dihedral groups $\Dih(A)$, where $A$ is an arbitrary abelian group. Recall that $\Dih(A)$ is the semidirect product of $A$ with a cyclic group of order $2$ acting on $A$ by inversion.

\begin{exa}\label{dih}{\rm
Given an arbitrary abelian group $A$, then $\Dih(A)$ is non-abelian if and only if $A$ is not elementary abelian.
In that case, a subgroup $\langle x \rangle$ is non-normal if and only if $x\not\in A$, and then
$C_A(x)$ is the subgroup of all elements of $A$ of order at most $2$.
Consequently, a non-abelian generalized dihedral group $\Dih(A)$ is an FCI-group if and only if it is a BCI-group, and this happens if and only if $A$ is of finite $2$-rank.}
\end{exa}

Recall that the rank of an abelian $p$-group is the dimension, as a vector space over
the field with $p$ elements, of the subgroup formed by the elements of order at most $p$. For an abelian $p$-group, to have finite rank is equivalent to requiring that the group is a direct sum of finitely many cyclic and quasicyclic groups \cite[4.3.13]{Ro}. If $A$ is an abelian group, the $p$-rank of $A$ is defined as the rank of
$A_p$, and the torsion-free rank, or $0$-rank, of $A$ is the cardinality of a maximal independent subset of elements of $A$ of infinite order.  More generally, a group $G$ is said to have finite (Pr\"ufer) rank $r$ if every finitely generated subgroup of $G$ can be generated by $r$ elements and $r$ is the least such integer. We write $\hat r(G)$ for $r$.

\vspace{8pt}

The classes of FCI- and BCI-groups are clearly closed under taking subgroups, but are not usually closed under taking quotients. For example, if $F$ is a free group, then $C_F(x)$ is cyclic for every $x\in F$, $x\ne 1$ (see, for instance, \cite[Proposition 2.19]{LS}), and in particular $F$ is an FCI-group. Since not all groups are FCI-groups, it follows that the property of being an FCI-group is not hereditary for quotients.

Next, consider $G=\Dih(A)$, where $A$ is torsion-free abelian of infinite $0$-rank.
Then $G$ is a BCI-group, as shown in Example \ref{dih}. On the other hand, if $N$ is the subgroup consisting of all fourth powers of elements of $A$, then $G/N=\Dih(A/N)$ is not a BCI-group, since $A/N$ is of infinite $2$-rank.

However, quotients by a finite normal subgroup have a better behaviour, as we see in the following proposition.

\begin{pro}
\label{G/N FCI-group}
Let $G$ be a group, and let $N$ be a finite normal subgroup of $G$.
Put $\overline G=G/N$. If $|C_G(x):\langle x \rangle|$ is finite, then
$$|C_{\overline G}(\overline x):\langle \overline x \rangle | \le |N| \, |C_G(x):\langle x \rangle|$$
is also finite. In particular, if $G$ is an FCI-group then so is $G/N$. Similarly for BCI-groups.
\end{pro}

\begin{proof}
Let $T$ be a subset of $G$ such that $\overline T$ is a transversal of
$\langle \overline x \rangle$ in $C_{\overline G}(\overline x)$,
and put $X=\{[x,t]\mid t\in T\}$. Since $X$ is contained in $N$ and $N$ is finite,
we can choose $t_1,\ldots,t_r \in T$, with $r\le |N|$, such that
\[
X = \{ [x,t_1], \ldots, [x,t_r] \}.
\]
Now if $t\in T$ then $[x,t]=[x,t_i]$ for some $i\in\{1,\ldots,r\}$, and consequently $tt_i^{-1}\in C_G(x)$.
If $S$ is a transversal of $\langle x \rangle$ in $C_G(x)$, then
$tt_i^{-1}\in \langle x \rangle s$ for some $s\in S$.
Hence $\overline t \in \langle \overline x \rangle \overline s \overline{t_i}$, and
$|\overline T | \le r|S|\le |N| |C_G(x):\langle x \rangle|$, which proves the result.
\end{proof}

One can readily check that the centre of any periodic FCI-group is finite. Then the previous proposition
yields that the central quotient of a periodic FCI-group is again an FCI-group. Similarly for BCI-groups.

The next result will play an important role in the investigation of infinite locally finite FCI-groups. Recall that the FC-centre of a group $G$ is the subgroup consisting of all elements whose conjugacy class in $G$ is finite, and that every infinite locally finite group always contains an infinite abelian subgroup (see \cite[14.3.7]{Ro}).

\begin{pro}
\label{infinite abelian}
Let $G$ be an infinite periodic FCI-group, and let $D$ be the FC-centre of $G$.
Then:
\begin{enumerate}
\item
$D$ consists of all elements $x\in G$ such that $\langle x \rangle \lhd G$, and every element of $G$ acts on $D$ by conjugation as a power automorphism.
In particular, $G'\leq C_G(D)$ and $D$ is a Dedekind group.
\item
If $G$ contains an infinite abelian subgroup $A$, then $A\le D$.
\end{enumerate}
\end{pro}

\begin{proof}
It is just a straightforward application of the definition.
\end{proof}

\section{Centralizers of power automorphisms}

A crucial step in our determination of locally finite FCI-groups is the analysis of the centralizer of a power automorphism of finite order of a periodic Dedekind group, which is the goal of this section. An automorphism of a group $G$ is a \emph{power automorphism\/} if it sends every element $x\in G$ to a power of $x$. Power automorphisms form an abelian subgroup of $\Aut G$, which we denote by $\PAut G$. For a periodic abelian group, we rely on the following result of Robinson \cite[Lemma 4.1.2]{Ro3}, which we write in a bit more detailed way for our convenience.

\begin{thm}
\label{PAut A robinson}
Let $A$ be an abelian group.
Then the following hold:
\begin{enumerate}
\item
If $A$ is periodic and $\varphi$ is an automorphism of $A$, then $\varphi\in\PAut A$ if and only if $\varphi_p\in\PAut A_p$ for every $p\in\pi(A)$.
As a consequence, the map
\[
\varphi \longmapsto (\varphi_p)_{p\in\pi(A)}
\]
yields an isomorphism between $\PAut A$ and the cartesian product
$\Cr_{p\in\pi(A)} \, \PAut A_p$.
\item
If $A$ is a $p$-group for a prime $p$, then $\PAut A \cong R^{\times}$, where $R=\Z_p$ if $\exp A=\infty$ and $R=\Z/p^n\Z$ if $\exp A=p^n<\infty$.
In both cases, an isomorphism can be obtained by sending every $t\in R^{\times}$ to the automorphism given by $a\mapsto a^t$ for all $a\in A$.
\end{enumerate}
\end{thm}

\begin{lem}
\label{cent of auto}
Let $A$ be an abelian $p$-group, where $p$ is a prime, and let $\varphi$ be a
power automorphism of $A$ of finite order $m>1$.
Then the following hold:
\begin{enumerate}
\item
If $m$ divides $p-1$, then $C_A(\varphi^k)=1$ for every $k\in\{1,\ldots,m-1\}$.
\item
If $m$ does not divide $p-1$ and $p>2$, then $\exp A<\infty$ and
$\Omega_1(A)\subseteq C_A(\varphi^k)$ for some $k\in\{1,\ldots,m-1\}$.
\item
If $p=2$ and $\exp A<\infty$, then $\Omega_1(A)\subseteq C_A(\varphi^k)$ for every
$k\in\{1,\ldots,m-1\}$.
\item
If $p=2$ and $\exp A=\infty$, then $m=2$ and $C_A(\varphi)=\Omega_1(A)$.
\end{enumerate}
\end{lem}

\begin{proof}
By (ii) of Theorem \ref{PAut A robinson}, the group $\PAut A$ is canonically isomorphic to the group of units $R^{\times}$, where
$R=\Z/p^n\Z$ or $\Z_p$, according as $\exp A=p^n<\infty$ or $\exp A=\infty$.

(i)
We have $p>2$ in this case.
Hence $R^{\times}=H\times J$, where the elements of $H$ correspond to the integers or
$p$-adic integers which are congruent to $1$ modulo $p$, and $J$ consists of the elements whose order divides $p-1$.
Let $k\in\{1,\ldots,m-1\}$, and let $t\in R^{\times}$ be such that $\varphi^k(a)=a^t$ for every $a\in A$.
Since the order of $\varphi^k$ divides $p-1$ and is greater than $1$, it follows that
$t\in J$ and $t\ne 1$.
Now, if $a\in A$ is of order $p$, then $\varphi^k(a)=a$ if and only if $t\in H$. This is impossible, since $H\cap J=1$.
Hence $\varphi^k$ does not fix any elements of order $p$ in $A$, and consequently
$C_A(\varphi^k)=1$.

(ii)
The torsion subgroup of $\Z_p^{\times}$ is of order $p-1$, since $p>2$.
Since $m$ does not divide $p-1$ in this case, $A$ must be of finite exponent, say $p^n$.
Thus $\PAut A\cong (\Z/p^n\Z)^{\times}$ is of order $p^{n-1}(p-1)$.
Then we can write $m=p^r k$ for some positive integers
$r$ and $k$.
Let $a\in A$ be an arbitrary element of order $p$.
Since $\langle a \rangle$ is $\varphi$-invariant and the order of $\varphi^k$ is a power of $p$, it follows that $\varphi^k$ fixes $a$.
Thus $\Omega_1(A)\subseteq C_A(\varphi^k)$.

(iii)
In this case, $\PAut A$ is a $2$-group.
Thus the argument given in the proof of (ii) applies to every power of $\varphi$, and $\Omega_1(A)\subseteq C_A(\varphi^k)$ for every $k$.

(iv)
The torsion subgroup of $\Z_2^{\times}$ has order $2$, and is generated by the inversion
automorphism.
Hence $m=2$ and $C_A(\varphi)=\Omega_1(A)$.
\end{proof}

\begin{lem}
\label{finite cent of autos}
Let $D$ be a periodic Dedekind group, and let $\varphi$ be a power automorphism of $D$ of finite order $m>1$.
Then the following two conditions are equivalent:
\begin{enumerate}
\item
$C_D(\varphi^k)$ is finite for every $k=1,\ldots,m-1$.
\vspace{5pt}
\item
$D_2$ is of finite rank and
$\displaystyle\prod_{p\in\pi_0\cup\pi_1} \, |D_p|$ is finite, where
\[
\pi_0 = \{ p\in \pi(D) \mid o(\varphi_p)<m \},
\]
and
\[
\pi_1 = \{ p\in \pi(D) \mid \text{$o(\varphi_p)=m$, $p>2$ and
$p\not\equiv 1\hspace{-10pt}\pmod m$} \}.
\]
\end{enumerate}
If these conditions hold then, for every $k=1,\ldots,m-1$, we have
\begin{equation}
\label{order centralizer phi^k}
|C_D(\varphi^k)|
\le
M \displaystyle\prod_{p\in\pi_0\cup\pi_1} \, |D_p|,
\end{equation}
where $M=|D_2|$ if $D_2$ is finite, and $M=2^{\hat r(D_2)}$ if $D_2$ is infinite.
\end{lem}

\begin{proof}
Observe that
\[
|C_D(\varphi^k)| = \prod_{p\in\pi(D)} \, |C_{D_p}(\varphi_p^k)|,
\quad
\text{for every $k=1,\ldots,m-1$,}
\]
and that the order of $\varphi_p$ divides $m$ for all $p$.
Let $p\in\pi(D)$ be an odd prime which does not lie in $\pi_0\cup\pi_1$.
Then the order of $\varphi_p$ is $m$, and $m$ divides $p-1$.
By applying (i) of Lemma \ref{cent of auto} to the automorphism $\varphi_p$ of $D_p$, it follows that
$C_{D_p}(\varphi_p^k)=1$ for $k=1,\ldots,m-1$.
Hence
\begin{equation}
\label{factorisation cent}
|C_D(\varphi^k)| = \prod_{p\in\pi_0\cup\pi_1\cup\{2\}} \, |C_{D_p}(\varphi_p^k)|,
\quad
\text{for every $k=1,\ldots,m-1$.}
\end{equation}

We can deduce immediately from here that (ii) implies (i), and also the bound in
(\ref{order centralizer phi^k}).
Clearly, it suffices to show that, provided that $D_2$ is infinite, we have
\[
| C_{D_2}(\varphi_2^k) | \le 2^{\hat r(D_2)}, \quad \text{for every $k=1,\ldots,m-1$.}
\]
In this case, we have $2\not\in\pi_0$, and consequently $o(\varphi_2)=m$.
On the other hand, $D_2$ is abelian and $\exp D_2=\infty$, since $D_2$ is infinite but of finite rank.
By (iv) of Lemma \ref{cent of auto} to $\varphi_2$, we obtain that $m=2$ and
$C_{D_2}(\varphi_2)=\Omega_1(D_2)$.
This completes the proof of the first implication.

Let us now prove that (i) implies (ii).
So we assume that $C_D(\varphi^k)$ is finite for every $k=1,\ldots,m-1$.
We first prove that both $\prod_{p\in\pi_0} \, |D_p|$ and $\prod_{p\in\pi_1} \, |D_p|$ are finite.
If $p\in\pi_0$ then $o(\varphi_p)<m$, and so $\varphi_p^k$ is the identity on $D_p$ for some $k\in\{1,\ldots,m-1\}$.
We conclude from
(\ref{factorisation cent}) that $\prod_{p\in\pi_0} \, |D_p|<\infty$.
Now if $p\in\pi_1$, then $p>2$ and $D_p$ is abelian.
By (ii) of Lemma \ref{cent of auto}, there exists $k\in\{1,\ldots,m-1\}$ such that
$\Omega_1(D_p)\subseteq C_{D_p}(\varphi_p^k)$.
It follows that $\prod_{p\in\pi_1} \, |\Omega_1(D_p)|<\infty$.
In particular, $\pi_1$ is finite, and so is $\Omega_1(D_p)$ for every $p\in\pi_1$.
But, again by (ii) of Lemma \ref{cent of auto}, we know that $\exp D_p<\infty$
if $p\in\pi_1$.
Consequently, $D_p$ is finite for every $p\in\pi_1$, and $\prod_{p\in\pi_1} \, |D_p|<\infty$, as desired.
Finally, since a power automorphism necessarily fixes all elements of order $2$, we have
$\Omega_1(D_2)\subseteq C_{D_2}(\varphi_2)$ and $D_2$ is of finite rank.
\end{proof}

\section{Locally finite FCI-groups}

We start our investigation of infinite locally finite FCI-groups by looking at the quotient with respect to their FC-centre.

\begin{pro}
\label{G/D cyclic}
Let $G$ be an infinite locally finite FCI-group, and let $D$ be the FC-centre of $G$.
Then $G/D$ is a finite cyclic group.
\end{pro}

\begin{proof}
Of course, we may assume that $D$ is a proper subgroup of $G$.
By Proposition \ref{infinite abelian}, $D$ is a periodic infinite Dedekind group.
Also, $D$ consists of all elements of $G$ which generate a normal subgroup of $G$.
As a consequence, $C_G(x)$ is finite for every $x\in G\smallsetminus D$.
Let $A$ be the centre of $D$. If $E$ is an infinite subgroup of $A$ then $C_G(E)=D$, and consequently $G/D$ embeds in $\PAut E$. In particular, $G/D$ is abelian.

We first assume that $D_p$ is infinite for some $p\in\pi(D)$. Then also $A_p$ is infinite, and $G/D$ embeds in $\PAut A_p\cong R^{\times}$, where $R=\Z_p$ or $\Z/p^n\Z$, by (ii) of Theorem \ref{PAut A robinson}. Unless we are in the case that $p=2$ and $\exp A_2<\infty$, the periodic elements in $R^{\times}$ form a cyclic subgroup, and consequently $G/D$ is cyclic, as desired. So it suffices to see that the assumption that $p=2$ and $\exp A_2<\infty$ leads to a contradiction.
In that case, consider an element $x\in G\smallsetminus D$.
If we apply (iii) of Lemma \ref{cent of auto} to the power automorphism of $A_2$ induced by $x$, we get $\Omega_1(A_2)\subseteq C_G(x)$, which is finite. It follows that $A_2$ is finite, contrary to our assumption.

Now we assume that $D_p$ is finite for every $p\in\pi(D)$. Since $D$ is infinite, it follows that $\pi(D)$ is infinite.
Let $xD\in G/D$ be an element of prime order, say $q$. Since $C_G(x)$ is finite, there exists a prime number $p_0$ such that $C_{D_p}(x)=1$ for every $p\ge p_0$. In other words, $x$ acts as a fixed-point-free automorphism of $D_p$ for $p\ge p_0$.
Then, by (ii) in 10.5.1 of \cite{Ro}, we have
\begin{equation}
\label{D_p=[D_p,x]}
D_p = \{ [d,x] \mid d\in D_p \}.
\end{equation}

Let $D^*=\prod_{p<p_0} \, D_p$, and write $\overline G$ for $G/D^*$.
Then $D^*$ is finite, and consequently $\overline G$ is also an FCI-group by
Proposition \ref{G/N FCI-group}.
On the other hand, by (\ref{D_p=[D_p,x]}), we have
\[
\overline D = \{ [\overline d,\overline x] \mid d\in D \},
\]
and consequently $\langle \overline x \rangle \ntriangleleft \overline G$ and $C_{\overline G}(\overline x)$ is finite.
Let now $g$ be an arbitrary element of $G$.
Since $G/D$ is abelian, we have $[g,x]\in D$.
Hence there exists $d\in D$ such that $[\overline g,\overline x]=[\overline d,\overline x]$, and then $\overline g\in C_{\overline G}(\overline x)\overline d$. Thus $\overline G=C_{\overline G}(\overline x) \overline D$, and so
$|\overline G:\overline D|$ is finite. We conclude that $G/D$ is finite.

Let us now consider, for every odd $p\in\pi(D)$, the homomorphism $\Phi_p$ from
$G/D$ to $\PAut A_p$ induced by the action of $G$ on $A_p$ by conjugation.
If $\ker\Phi_p$ is trivial for some $p$, then we can conclude that $G/D$ is cyclic, since
$\PAut A_p$ is cyclic for odd $p$. Hence we may assume that $\ker\Phi_p$ is always non-trivial.
Since $G/D$ is finite and $\pi(D)$ is infinite, there exists an infinite subset $\pi^*$ of
$\pi(D)$ such that $\ker\Phi_p$ is the same subgroup, say $K/D$, for every $p\in\pi^*$.
If we choose an element $x\in K\smallsetminus D$, it follows that $A_p\subseteq C_G(x)$
for every $p\in\pi^*$. Thus $C_G(x)$ is infinite, which is a contradiction.
This completes the proof.
\end{proof}

Notice that, in \cite{Sha}, Shalev considered locally finite groups in which the centralizer of each element is either finite or of finite index. Clearly, such groups are FCI-groups. Then, Proposition 2.5 of \cite{Sha} shows that the FC-centre of any locally finite FCI-group is of finite index. In the previous proposition we gave a direct proof of this fact. However, contrary to \cite{Sha}, our proof does not depend on the classification of finite simple groups.

\begin{cor}
\label{metabelian}
Let $G$ be an infinite locally finite FCI-group. Then $G$ is metabelian.
\end{cor}

\begin{proof}
Let $D$ be the FC-centre of $G$. Then $G'\le C_G(D)$, by (i) of Proposition \ref{infinite abelian}, and so the result follows from Proposition \ref{G/D cyclic}.
\end{proof}

Before proceeding, we need the following lemma, which can be easily checked.

\begin{lem}
\label{PAut Q8}
An automorphism of the quaternion group is a power automorphism if and only if it is inner.
\end{lem}

We are now ready to give a detailed description of the structure of infinite locally finite FCI-groups.

\begin{thm}
\label{determination locally finite FCI-groups}
A group $G$ is an infinite locally finite FCI-group if and only either
\begin{enumerate}
\item $G$ is an infinite periodic Dedekind group, or
\item $G=\langle g,D \rangle$, where $D$ is an infinite periodic Dedekind group with $D_2$ of finite rank, and $g$ acts on $D$ as a power automorphism $\varphi$ of order $m>1$ such that
$|G:D|=m$ and
\begin{equation}
\label{condition for locally finite}
\prod_{p\in\pi_0\cup\pi_1} \, |D_p| < \infty,
\end{equation}
where
\[
\pi_0 = \{ p\in \pi(D) \mid o(\varphi_p)<m \},
\]
and
\[
\pi_1 = \{ p\in \pi(D) \mid \text{$o(\varphi_p)=m$, $p>2$ and
$p\not\equiv 1\hspace{-10pt}\pmod m$} \}.
\]
\end{enumerate}
Furthermore, if {\rm (ii)} holds, then for every $\langle x \rangle \ntriangleleft G$ we have
\begin{equation}
\label{bound C_G(x) typical locally finite BCI-group}
|C_G(x)|
\le
m M \prod_{p\in\pi_0\cup\pi_1} \, |D_p|,
\end{equation}
where $M=|D_2|$ if $D_2$ is finite, and $M=2^{\hat r(D_2)}$ if $D_2$ is infinite.
\end{thm}

\begin{proof}
We first prove the `only if' part of the statement. Let $G$ be an infinite locally finite FCI-group and let $D$ be the FC-centre of $G$. By Proposition \ref{infinite abelian},
$D$ consists of all elements of $G$ which generate a normal subgroup of $G$, and $D$ is an infinite periodic Dedekind group.
We assume that $D$ is a proper subgroup of $G$, and show that (ii) holds.

Let us write $D=Q\times A$, where $Q$ is either trivial or isomorphic to the quaternion group, and $A$ is infinite abelian.
Then $C_G(A)=D$, since $G$ is an FCI-group. We also know from Proposition \ref{G/D cyclic} that $G/D$ is cyclic.
We claim that $G=DC_G(Q)$. This is obvious if $Q=1$, so assume that $Q$ is the quaternion group.
Observe that
\begin{equation}
\label{inequalities}
|D:C_D(Q)| = |D:C_G(Q)\cap D| \le |G:C_G(Q)| \le |\PAut Q|.
\end{equation}
Now $C_D(Q)=Z(Q)\times A$ has index $4$ in $D$, and by Lemma \ref{PAut Q8},
$\PAut Q=\Inn Q$ is of order $4$.
Hence the first inequality in (\ref{inequalities}) is an equality, and as a consequence
$G=DC_G(Q)$.

Let $m>1$ be the order of $G/D$, and let us choose a generator $gD$ of $G/D$.
Since $G=DC_G(Q)$, we may assume that $g\in C_G(Q)$. Then $g$ induces by conjugation a power automorphism $\varphi$ of $D$ which coincides with the identity on $Q$. Thus $o(\varphi)=o(\varphi_{|A})$. Since $C_G(A)=D$, we have a natural embedding of $G/D$ in $\PAut A$, and so $o(\varphi_{|A})=o(gD)=m$.
Thus $o(\varphi)=m$. Now if a power $g^k$ does not lie in $D$, then $\langle g^k \rangle \ntriangleleft G$, and $C_G(g^k)$ is finite.
Thus $C_D(\varphi^k)$ is finite for every $k=1,\ldots,m-1$, and we only need to apply
Lemma \ref{finite cent of autos} to get (ii).

Conversely, let $G$ be as in part (ii). The group $G$ is clearly locally finite, and every element of $D$ generates a normal subgroup of $G$, since $D$ is a Dedekind group and $\varphi$ is a power automorphism of $D$.
Thus it suffices to prove that the bound in (\ref{bound C_G(x) typical locally finite BCI-group}) holds for every $x\in G\smallsetminus D$.
Let us write $x=g^kd$ with $k\in\{1,\ldots,m-1\}$ and $d\in D$.
Assume first that $D_2$ is infinite.
Then $D$ is abelian, since $D_2$ is of finite rank.
Consequently
\begin{align*}
|C_G(x)|
&\le
|G:D| \, |C_D(x)| = m \, |C_D(g^k)| = m \,  |C_D(\varphi^k)|
\\
&\le m 2^{\hat r(D_2)} \prod_{p\in\pi_0\cup\pi_1} \, |D_p|,
\end{align*}
by Lemma \ref{finite cent of autos}.
This proves the result if $D_2$ is infinite.
On the other hand, if $D_2$ is finite, put $A=\prod_{p>2} \, D_p$.
Then $A\le Z(D)$, and
\[
|C_G(x)|\le |G:A| \, |C_A(x)| = m \, |D_2| \, |C_A(g^k)| = m \, |D_2| \,  |C_A(\varphi^k)|.
\]
Now the proof of Lemma \ref{finite cent of autos} shows that $C_A(\varphi^k)$ is contained in
the product $\prod_{p\in(\pi_0\cup\pi_1)\smallsetminus\{2\}} \, D_p$, and
(\ref{bound C_G(x) typical locally finite BCI-group}) follows.
\end{proof}

The following corollary is an immediate consequence of Theorem \ref{determination locally finite FCI-groups}.

\begin{cor}
\label{locally finite FCI->BCI}
Let $G$ be a locally finite group.
Then the following conditions are equivalent:
\begin{enumerate}
\item
$G$ is an FCI-group.
\item
$G$ is a BCI-group.
\item
There exists $n\in\N$ such that $|C_G(x)|\le n$ whenever $\langle x \rangle \ntriangleleft G$.
\end{enumerate}
\end{cor}

Recall from the theory of cyclic extensions \cite[Section III.7]{Za} that, given a group $N$ and an automorphism $\varphi$ of $N$, all extensions $G=\langle g,N \rangle$ of $N$ by a finite cyclic group $C_m$, where $g$ acts on $N$ via $\varphi$, are obtained by choosing the power $g^m$ to be an element $n\in N$ such that:
\begin{enumerate}
\item
$\varphi(n)=n$.
\item
$\varphi^m$ coincides with the inner automorphism of $N$ induced by $n$.
\end{enumerate}
In particular, if $\varphi$ is of order $m$, these two conditions reduce to
$n\in C_{Z(N)}(\varphi)$.

\begin{rmk}{\rm
According to Theorem \ref{determination locally finite FCI-groups}, if we want to produce an infinite locally finite FCI-group $G$, we can proceed as follows.
We start by taking an infinite periodic Dedekind group $D$ with $D_2$ of finite rank, and then we choose power automorphisms $\varphi_p$ of the Sylow $p$-subgroups $D_p$ in such a way that the least common multiple $m$ of the orders of all the $\varphi_p$ is finite and greater than $1$, and that condition (\ref{condition for locally finite}) holds.
Then we can use the natural isomorphism between $\PAut D$ and $\Cr_{p\in\pi(D)}\, \PAut D_p$ to define a power automorphism $\varphi$ of $D$, and we can construct $G$ as an extension of $D$ by the cyclic group of order $m$, by using the method explained above.
Theoretically, selecting the automorphisms $\varphi_p$ is an easy task: if $D_p$ is abelian (in particular for all $p>2$) we know that $\PAut D_p$ is isomorphic to $(\Z/p^n\Z)^{\times}$ if
$\exp D_p=p^n<\infty$ and to $\Z_p^{\times}$ if $\exp D_p=\infty$, and the structure of these groups of units is well-known.
In practice, we have to face the problem that no general procedure is known to obtain a primitive root modulo $p^n$ (i.e.\ a generator of $(\Z/p^n\Z)^{\times}$) for odd $p$ in polynomial time.
As for the choice of $\varphi_2$ when $D_2$ is non-abelian, see the next remark.
}
\end{rmk}

\begin{rmk}
{\rm
\label{phi2 is inversion}
If $D_2$ is infinite in part (ii) of Theorem \ref{determination locally finite FCI-groups}, then $D_2$ must be abelian of infinite exponent, since it is of finite rank.
Thus $\PAut D_2\cong \Z_2^{\times}$, and since
$\varphi_2$ is of finite order, it must be either the identity or inversion.
Now observe that $2\not\in\pi_0$, by condition (\ref{condition for locally finite}).
Thus we can conclude that if $D_2$ is infinite then $\varphi_2$ is necessarily the inversion automorphism, and that $m=2$.
}
\end{rmk}

\section{Periodic BCI-groups}

This section is devoted to periodic BCI-groups.
Given a periodic BCI-group $G$, we prove that if $G$ is also locally graded, then it is locally finite.
We need the following two lemmas.

\begin{lem}[{\cite[Proposition 1]{Lu}}]
\label{pq}
Let $G$ be a finite soluble group.
If $p$, $q$, and $r$ are distinct primes dividing the order of $G$, then $G$ contains an element of order the product of two of these primes.
\end{lem}

\begin{lem}[{\cite[page 263]{Sh}}]
\label{involution}
Let $G$ be a periodic group having an involution with finite centralizer.
Then $G$ is locally finite.
\end{lem}

Recall that a quotient of a locally graded group need not be locally graded. However, it is straightforward to see that the quotient of a locally graded group over a finite normal subgroup is again locally graded.

\begin{thm}
\label{graded}
Every locally graded periodic BCI-group is locally finite.
\end{thm}

\begin{proof}
Let $G$ be a finitely generated locally graded periodic BCI-group, and suppose for a contradiction that $G$ is infinite.
By hypothesis, there exists a positive integer $n$ such that
$|C_G(x):\langle x \rangle|\le n$ whenever $\langle x \rangle\ntriangleleft G$.

Let $D$ be the FC-centre of $G$.
By (i) of Proposition \ref{infinite abelian}, we have $G'\le C_G(D)$.
If $G'$ is a Dedekind group, then $G$ is soluble and, being periodic and finitely generated, $G$ is finite.
This is a contradiction.
Thus there exists $x\in G'$ such that $\langle x \rangle \ntriangleleft G$.
Hence $C_G(x)$ is finite and so also $D$ is finite.
Then $G/D$ is an infinite locally graded group and, by Proposition \ref{G/N FCI-group}, it is also a BCI-group. In particular
$G/D$ is a counterexample with trivial FC-centre, and without loss of generality we may assume that $\l x\r \ntriangleleft G$ for every $x\in G$, $x\neq 1$.
Thus, by Lemma \ref{involution}, $G$ has no involutions. As a consequence every finite quotient of $G$ is of odd order, and by the Feit-Thompson Theorem, is soluble.

Now let $p,q>n$ be prime numbers. If $x\in G$ is of order $pq$, then
\[
p = |\langle x \rangle:\langle x^p \rangle| \le |C_G(x^p):\langle x^p \rangle| \le n.
\]
Hence $G$ does not have any elements of order $pq$, and the same is true for all quotients of $G$. Let $R$ be the finite residual of $G$, i.e.\ the intersection of all subgroups of $G$ of finite index, and let us show that $\pi(G/R)$ is finite.
Suppose, by way of contradiction, that $\pi(G/R)$ is infinite.
Then we can find elements $x,y,z\in G$ such that $o(xR)=p$, $o(yR)=q$ and $o(zR)=r$ are three distinct primes greater than $n$.
By the definition of $R$, there exists $N\lhd G$ of finite index such that $x,y,z\not\in N$.
Then $o(xN)=p$, $o(yN)=q$ and $o(zN)=r$.
But, as mentioned above, $G/N$ is soluble and does not contain any elements of order $pq$, $pr$, or $qr$.
This is impossible, according to Lemma \ref{pq}.

Next, given $p\in\pi(G/R)$, let us consider an arbitrary non-trivial $p$-element $gR\in G/R$, say of order $p^k$.
If $k>n$ then
\[
|C_G(g^{p^n}):\langle g^{p^n} \rangle| \ge |\langle g \rangle:\langle g^{p^n} \rangle|
= p^n > n,
\]
a contradiction.
Hence the order of all $p$-elements of $G/R$ is at most $p^n$, and by the finiteness of $\pi(G/R)$, it follows that $G/R$ is a group of finite exponent. Finally, since $G/R$ is finitely generated and residually finite, Zelmanov's positive solution to the Restricted Burnside Problem yields that $G/R$ is finite. It follows that $R$ is finitely generated. Since $G$ is locally graded, there exists $K<R$ such that $|R:K|$ is finite, and then also $|G:K|$ is finite. This is the final contradiction, because $R$ is contained in all subgroups of $G$ of finite index.
\end{proof}

\noindent
\textit{Acknowledgment\/}. We wish to thank Professor D.\,J.\,S.\ Robinson for interesting discussions and helpful comments.

\end{document}